\documentclass[10pt]{amsart}
\usepackage{amsmath,amssymb,mathtools}
\usepackage[T1]{fontenc}
\usepackage{lmodern}
\usepackage{hyperref}
\hypersetup{hidelinks}
\newtheorem{theorem}{Theorem}[section]
\newtheorem{proposition}[theorem]{Proposition}
\newtheorem{lemma}[theorem]{Lemma}

\theoremstyle{definition}

\theoremstyle{remark}
\newtheorem{remark}[theorem]{Remark}
\newcommand{\R}{\mathbb R}
\newcommand{\ii}{\sqrt{-1}}
\newcommand{\tr}{\operatorname{tr}}
\newcommand{\Lev}{\mathcal L}
\newcommand{\Ric}{\operatorname{Ric}}
\newcommand{\Rm}{\operatorname{Rm}}
\newcommand{\ddc}{dd^c}
\newcommand{\vol}{d\mu_g}
\newcommand{\Herm}{\operatorname{Herm}}
\newcommand{\Lip}{\operatorname{Lip}}
\newcommand{\norm}[1]{\lVert#1\rVert}
\numberwithin{equation}{section}
\title[Ricci pinching on K\"ahler manifolds]{Ricci pinching on complete K\"ahler manifolds with nonnegative bisectional curvature}
\author{Hanzhang Yin}
\address{School of Mathematics, Harbin Institute of Technology, Harbin, Heilongjiang 150001, China}
\email{3199877769@qq.com}
\email{YinHZ@hit.edu.cn}
\thanks{The author is supported by the Research Start-up Fund of Harbin Institute of Technology.}
\date{}
\begin{document}
\begin{abstract}
We prove that a connected complete K\"ahler manifold of complex dimension
at least two with nonnegative holomorphic bisectional curvature and
uniformly pinched Ricci curvature is compact or flat. This removes the
Euclidean volume growth assumption from the flatness theorem of Liu and
Ni \cite{LiuNi} for uniform Ricci pinching. We first construct a Lipschitz
plurisubharmonic function with finite total Monge--Amp\`ere mass by solving
a global Bellman equation. Weighted estimates for a holomorphic canonical
section then give a finite top Ricci integral. Ricci pinching implies
that the Ricci form belongs to $L^n$, and its Hodge heat evolution has
zero top integral by a Stokes argument and scalar heat decay.
\end{abstract}
\maketitle

\section{Introduction}\label{sec:introduction}

The Bonnet--Myers theorem gives compactness of a complete Riemannian
manifold whose Ricci tensor has a uniform positive lower bound. A related
question asks whether a pointwise comparison of the smallest and largest
Ricci eigenvalues implies compactness or flatness under suitable
curvature assumptions. In real dimension three, Hamilton conjectured that a complete connected
Riemannian manifold with nonnegative, uniformly pinched Ricci curvature
is compact or flat, see \cite[Conjecture~3.39]{CLN}.
This conjecture was solved through the work of Lott \cite{Lott},
Deruelle--Schulze--Simon \cite{DSS}, and Lee--Topping \cite{LeeTopping}.

Liu and Ni \cite[Theorem~1.1]{LiuNi} proved that a complete noncompact
K\"ahler manifold of complex dimension at least two is flat if it has
nonnegative bisectional curvature, Euclidean volume growth, and
\eqref{eq:pinching}.
We remove their volume growth assumption, and prove the following result in K\"ahler geometry.

\begin{theorem}\label{thm:main}
Let $(M^n,g,J)$ be a connected complete K\"ahler manifold of complex
dimension $n\geq2$, with Riemannian metric $g$, complex structure $J$,
and nonnegative holomorphic bisectional curvature.
Suppose that there is a constant $\kappa>0$ such that
\begin{equation}\label{eq:pinching}
 \Ric_{\min}(x)\geq\kappa\Ric_{\max}(x)
 \qquad\text{for all }x\in M.
\end{equation}
Then $M$ is compact or flat.
\end{theorem}

Here $\Ric$ denotes the Ricci tensor, and $\Ric_{\min}(x)$ and
$\Ric_{\max}(x)$ are the smallest and largest eigenvalues, respectively,
of the endomorphism $g_x^{-1}\Ric_x:T_xM\to T_xM$.

The central difficulty is to prove
\begin{equation}\label{eq:intro-finite}
 \int_M\rho^n<\infty,
\end{equation}
where $\rho(X,Y)=\Ric(JX,Y)$ is the Ricci form and $\rho^n$ is its
$n$th wedge power. We construct a continuous plurisubharmonic function
$v$ such that
\begin{equation}\label{eq:intro-potential}
 \Lip(v)\leq1,\qquad \int_M(\ddc v)^n<\infty.
\end{equation}
Here $\Lip(v)$ is the global Lipschitz constant with respect to the
distance of $g$, $\ddc v=\ii\partial\bar\partial v$, and the top power
is the Bedford--Taylor measure. We use PSH as an abbreviation for
plurisubharmonic. The limit also satisfies a trace inequality with a
nonzero source of compact support, which excludes the pluriharmonic
case even when its total Monge--Amp\`ere mass is zero; see
\eqref{eq:upper-tests} and Proposition~\ref{prop:finite-ma}.

To construct $v$, we regularize the Bellman minimum in
\eqref{eq:bellman}, add a positive trace term, and solve the resolvent
equation \eqref{eq:resolvent}. The global barriers and comparison
principle in Lemmas~\ref{lem:barriers} and \ref{lem:comparison} give a
smooth solution by Proposition~\ref{prop:resolvent}. A Bernstein estimate
first bounds its gradient, and Proposition~\ref{prop:sharp-gradient}
then gives $|\nabla u|\leq1$. The normalized solutions converge locally
uniformly to $v$. Lemma~\ref{lem:limit-tests} and
Proposition~\ref{prop:finite-ma} give \eqref{eq:intro-potential}.

Next, the heat regularizations in
Proposition~\ref{prop:strict-approximation} are strictly PSH, have
gradient at most one, and differ from $v$ by a common bounded amount.
The parallel kernel alternatives of Ni and Tam reduce strictness to
Lemma~\ref{lem:no-product}: a proper nonzero parallel complex subbundle,
together with Ricci pinching, would imply a uniform positive Ricci lower
bound and hence compactness. A holomorphic canonical section gives the
identity \eqref{eq:energy}. The weighted mixed estimates in
Lemma~\ref{lem:weighted-reduction} and
Proposition~\ref{prop:top-comparison} then imply
\eqref{eq:intro-finite}, by the successive limits in
Proposition~\ref{prop:finite-ricci-criterion}.

Finally, Ricci pinching and \eqref{eq:intro-finite} imply that $\rho$
belongs to $L^n(M)$, as shown in Lemma~\ref{lem:finite-mass-flat}.
The Hodge heat construction of Ni and Tam
\cite[Theorem~4.1]{NiTam13} and the scalar heat estimates in
Lemmas~\ref{lem:scalar-decay} and \ref{lem:integrated-gradient} apply.
The Stokes identity in Lemma~\ref{lem:critical-stokes} preserves the top
integral at each fixed heat time, while scalar heat decay forces it to
vanish; see Proposition~\ref{prop:critical-vanishing}.
This argument follows the critical vanishing method of Liu and Ni
\cite[Section~4]{LiuNi}. The construction of
\eqref{eq:intro-potential} and the proof of \eqref{eq:intro-finite}
remove the volume growth assumption.

Section~\ref{sec:preliminaries} gives the conventions and geometric
results used in the proof. Section~\ref{sec:construction} constructs the
PSH function with finite Monge--Amp\`ere mass.
Section~\ref{sec:ricci-mass} proves finiteness of the Ricci integral.
Section~\ref{sec:critical} proves vanishing under critical integrability,
and Section~\ref{sec:conclusion} completes the proof of
Theorem~\ref{thm:main} by passing to the universal cover.

\section{Geometric and analytic preliminaries}\label{sec:preliminaries}

\subsection{Conventions}
All manifolds are nonempty, smooth, Hausdorff, second countable, and without
boundary. Write $d$ for the Riemannian distance of $g$, $\nabla$ for
its Levi--Civita connection, $\Rm_g$ for its curvature tensor, and
$\nabla^2f$ for the covariant Hessian of a function $f$. Completeness
refers to $d$. Put
\[
 \omega(X,Y)=g(JX,Y),\qquad \rho(X,Y)=\Ric(JX,Y),\qquad
 \vol=\frac{\omega^n}{n!}.
\]
Here $\omega$ is the K\"ahler form, $\rho$ is the Ricci form, and
$\vol$ is the Riemannian volume measure. For real functions, use
\begin{equation}\label{eq:dc-convention}
 d^cf=-\frac12df\circ J,\qquad \ddc f=\ii\partial\bar\partial f.
\end{equation}
The operators $\partial$ and $\bar\partial$ are the Dolbeault
operators. A function is PSH if its restriction to each holomorphic
disc is subharmonic; for a smooth function this is equivalent to
$\ddc f\geq0$. It is strictly PSH if $\ddc f>0$, and pluriharmonic
if $\ddc f=0$ in the sense of distributions.
A real $(1,1)$-form $\alpha$ is nonnegative when
$\alpha(X,JX)\geq0$. All inequalities between such forms use this order.
The scalar Laplacian is the real Laplace--Beltrami operator
\[
 \Delta f=\tr_g\nabla^2 f,
\]
with nonpositive $L^2$ spectrum. We write $\Lambda_\omega\alpha$
for the trace of a real $(1,1)$-form $\alpha$ with respect to $\omega$.
Thus
\begin{equation}\label{eq:trace-real}
 \Lambda_\omega\ddc f=\frac12\Delta f.
\end{equation}

At a point choose a real orthonormal frame
$e_1,\ldots,e_n,Je_1,\ldots,Je_n$. The Hermitian matrix used for the
Bellman operator is
\begin{equation}\label{eq:levi-matrix}
 (\Lev f)_{j\bar k}=\frac14\bigl(
  \nabla^2 f(e_j,e_k)+\nabla^2 f(Je_j,Je_k)
  +\ii[\nabla^2 f(e_j,Je_k)-\nabla^2 f(Je_j,e_k)]\bigr).
\end{equation}
The eigenvalues of $\ddc f$ relative to
$\omega$ are twice those of $\Lev f$. Consequently
\begin{equation}\label{eq:levi-normalizations}
 \tr\Lev f=\frac14\Delta f,\qquad
 (\ddc f)^n=2^n n!\det(\Lev f)\vol.
\end{equation}
These formulas also specify the normalization when the last identity
is interpreted as an equality of Bedford--Taylor measures.
The parallel orthogonal complex structure has vanishing Nijenhuis
tensor; smooth holomorphic coordinates therefore exist by the
integrability theorem \cite[Theorem~2]{Malgrange}. Locally smooth
adapted orthonormal frames can be obtained using the smooth positive
square root of a positive matrix \cite[Theorem~2.6]{DPP}. We use the
exterior derivative $d$ and the usual wedge product throughout,
with the conventions of \cite[Chapters~14 and 16]{LeeSM}.

We use the metric norm on real $(1,1)$-forms. If $A$ is the matrix of
such a form in the convention of \eqref{eq:levi-matrix}, this norm is
$2\norm{A}_{\mathrm{HS}}$, where $\norm{A}_{\mathrm{HS}}=(\tr(A^*A))^{1/2}$
is the Hilbert--Schmidt norm and $A^*$ is the conjugate transpose.
In particular, a nonnegative form satisfies
\begin{equation}\label{eq:form-order-bound}
 0\leq\alpha\leq |\alpha|\omega.
\end{equation}
For nonnegative forms $\alpha_1,\ldots,\alpha_n$, it follows that
\begin{equation}\label{eq:mixed-norm}
 0\leq\frac{\alpha_1\wedge\cdots\wedge\alpha_n}{\vol}
 \leq n!\prod_{j=1}^n|\alpha_j|.
\end{equation}
We use the usual positivity of products of nonnegative $(1,1)$-forms
\cite[Section~1.C]{Demailly}. All mixed forms in the proof are products
of smooth closed forms on $M$.

\subsection{Exhaustions and cutoffs}
The distance heat regularization of Ni and Tam
\cite[Proposition~3.1]{NiTam13} supplies, under nonnegative bisectional
curvature, a nonnegative smooth proper function $b$ with
\begin{equation}\label{eq:exhaustion}
 |\nabla b|\leq1,\qquad -C\omega\leq\ddc b\leq C\omega.
\end{equation}
If $\theta$ is a fixed smooth cutoff on the
line, constant near zero and outside a bounded interval, then
$\chi_R=\theta(b/R)$ can be chosen so that, for $R\geq1$,
\begin{equation}\label{eq:cutoffs}
\begin{gathered}
 0\leq\chi_R\leq1,\quad \chi_R\in C_c^\infty(M),\quad
 \chi_R\longrightarrow1\text{ pointwise},\\
 |\nabla\chi_R|\leq\frac{C}{R},\quad
 -\frac{C}{R}\omega\leq\ddc\chi_R\leq\frac{C}{R}\omega.
\end{gathered}
\end{equation}
Constants denoted by $C$ may change from one estimate to the next;
their dependence is specified when uniformity is needed.
For every integer $N\geq2$, differentiation gives
\begin{equation}\label{eq:cutoff-powers}
 |d(\chi_R^N)|\leq\frac{C_N}{R}\chi_R^{N-1},\qquad
 -\frac{C_N}{R}\chi_R^{N-2}\omega
 \leq\ddc(\chi_R^N)\leq\frac{C_N}{R}\chi_R^{N-2}\omega.
\end{equation}
For the Bernstein argument we also use squared cutoffs $\eta_R$ with
$|\nabla\eta_R|^2\leq CR^{-2}\eta_R$ and the same complex-Hessian bound.
All constants for a finite range of $N$ can be chosen uniformly over
that range.

\subsection{Pinching and products}
Nonnegative holomorphic bisectional curvature implies $\Ric\geq0$.
If $\Ric\equiv0$, all bisectional curvatures vanish by tracing their
nonnegative values. Polarization of the K\"ahler curvature tensor then
shows that $M$ is flat. Thus nonflatness gives a point with a positive Ricci
direction. Under \eqref{eq:pinching}, there are a point $p$ and $c>0$
such that $\Ric_p\geq cg_p$.

\begin{lemma}\label{lem:no-product}
Suppose that $M$ is noncompact, simply connected, and nonflat, and
satisfies the hypotheses of Theorem~\ref{thm:main}. Then its tangent
bundle has no proper nonzero parallel orthogonal complex subbundle.
\end{lemma}
\begin{proof}
Such a subbundle and its orthogonal complement give a global product
$M_1\times M_2$ by de Rham decomposition
\cite[Theorem~8.1, p.~172]{KN}. Both factors have positive dimension.
At the point $p=(p_1,p_2)$ above, every unit vector in either factor has
Ricci curvature at least $c$. Fix a unit vector in $T_{p_2}M_2$. Combine
it with any unit vector at any point of $M_1$. Since the two factors
can be chosen independently and product Ricci is additive, pinching at
the resulting product point shows that every unit direction in $M_1$
has Ricci curvature at least $\kappa c$. Interchanging the factors gives
the same conclusion for $M_2$. Orthogonal decomposition therefore gives
$\Ric\geq\kappa cg$ on the product. Bonnet--Myers makes $M$ compact,
which is impossible.
\end{proof}

\begin{lemma}\label{lem:source-data}
Under the hypotheses of Lemma~\ref{lem:no-product}, there are
$u_0,h\in C^\infty(M)$ and $r:M\to[0,\infty)$ such that
\begin{equation}\label{eq:initial-data}
 \Lev u_0\geq0,\quad |\nabla u_0|\leq1,\quad
 h\in C_c^\infty(M),\quad h\geq0,\quad h(p)>0
\end{equation}
at some point $p$, and
\begin{equation}\label{eq:source-data}
 h\leq\tr\Lev u_0,\qquad |\nabla h|\leq r/4,\qquad \Ric\geq rg.
\end{equation}
\end{lemma}
\begin{proof}
The noncompact factor in the simply connected splitting of
\cite[Theorem~4.2(i)]{NiTam03} supports a smooth strictly PSH function
with bounded gradient. Pull it back to $M$ and rescale. This gives
$\Lev u_0\geq0$, $|\nabla u_0|\leq1$, and $\tr\Lev u_0>0$ everywhere.
Choose $p$ with $\Ric_p>0$. On a relatively compact neighborhood $U$
of $p$, choose $a,c>0$ with $\tr\Lev u_0\geq a$ and $\Ric\geq cg$.
Take $\psi\in C_c^\infty(U)$ with $0\leq\psi\leq1$ and $\psi(p)=1$.
If $A\geq\sup|\nabla\psi|$, choose
\[
 0<t\leq\min\{a,c/[4(A+1)]\},\qquad h=t\psi.
\]
Set $r=c$ on $U$ and $r=0$ outside $U$. The gradient of $h$ vanishes
off its compact support, and $\Ric\geq0$ everywhere. These observations
give all assertions.
\end{proof}

\section{PSH functions with finite Monge--Amp\`ere mass}\label{sec:construction}

We solve a uniformly elliptic approximation of a Bellman equation and
prove a gradient bound independent of its parameters. The normalized
limit will satisfy the measure estimate in Proposition~\ref{prop:finite-ma}.

\subsection{The regularized Bellman operator}\label{sec:operator}

Let $\Herm(n)$ be the real vector space of complex Hermitian
$n\times n$ matrices, let $I$ be the identity matrix, and write
$\tr$ and $\det$ for the trace and determinant. Matrix inequalities
refer to the Hermitian order. For $0<\varepsilon\leq1$, define the
compact nonempty set
\[
 \mathcal B_\varepsilon=
 \{B\in\Herm(n):\det B=1,\ \varepsilon I\leq B\leq\varepsilon^{-1}I\}
\]
and
\begin{equation}\label{eq:bellman}
 F_\varepsilon(A)=\min_{B\in\mathcal B_\varepsilon}
             \frac1n\tr(AB).
\end{equation}
We extend $F_\varepsilon$ to all complex matrices by replacing $A$
with its Hermitian part $(A+A^*)/2$, or equivalently by taking the
real part of the trace in \eqref{eq:bellman}. It is concave, unitarily
invariant, positively homogeneous, and increasing in the Hermitian order.
If $P\geq0$, then
\begin{equation}\label{eq:bellman-ellipticity}
 \frac{\varepsilon}{n}\tr P
 \leq F_\varepsilon(A+P)-F_\varepsilon(A)
 \leq\frac{\varepsilon^{-1}}n\tr P.
\end{equation}
Both inequalities follow by evaluating at a minimizing matrix for one
of the two arguments. Compactness of $\mathcal B_\varepsilon$ justifies
these choices. The same observation gives a Lipschitz constant of order
$\varepsilon^{-1}$ in any fixed finite-dimensional matrix norm.

Equip the space of complex matrices with its real Hilbert--Schmidt
inner product, and convolve $F_\varepsilon$ with a smooth
nonnegative radial kernel of mass one supported in the ball of radius
$\varepsilon^3$. Write the resulting operator as
$\widetilde F_\varepsilon$. Radiality preserves unitary invariance;
convolution preserves concavity and \eqref{eq:bellman-ellipticity}.
There is a dimensional constant $c_n$ such that
\begin{equation}\label{eq:mollifier-error}
 |\widetilde F_\varepsilon(A)-F_\varepsilon(A)|
 \leq c_n\varepsilon^2.
\end{equation}
Since $I\in\mathcal B_\varepsilon$ and the convolution kernel is
centered,
\begin{equation}\label{eq:bellman-trace}
 \widetilde F_\varepsilon(A)\leq\frac1n\tr A.
\end{equation}

Fix $\tau>0$ and set
\begin{equation}\label{eq:G}
 G_{\varepsilon,\tau}(A)=\widetilde F_\varepsilon(A)+\tau\tr A.
\end{equation}
Its derivative is represented on Hermitian matrices by a positive
Hermitian matrix $a(A)$ satisfying
\begin{equation}\label{eq:a-bounds}
 (\varepsilon/n+\tau)I\leq a(A)\leq(\varepsilon^{-1}/n+\tau)I.
\end{equation}
Thus $G_{\varepsilon,\tau}(\Lev u)$ is independent of the adapted
orthonormal frame and defines an operator on $M$.

We will also use the supporting hyperplane representation. Put
\[
 c(A)=G_{\varepsilon,\tau}(A)-DG_{\varepsilon,\tau}(A)[A].
\]
Concavity gives
\begin{equation}\label{eq:supporting}
 G_{\varepsilon,\tau}(H)
 =\inf_A\{DG_{\varepsilon,\tau}(A)[H]+c(A)\},\qquad
 |c(A)|\leq c_n\varepsilon^2.
\end{equation}
The supporting hyperplane inequality gives the infimum
formula, with equality for $A=H$. To obtain the intercept bound, compare
$G_{\varepsilon,\tau}$ with the positively homogeneous function
$F_\varepsilon+\tau\tr$ using \eqref{eq:mollifier-error}. The supporting
inequality at zero gives $c(A)\geq-c_n\varepsilon^2$. Applying it at
$tA$, dividing by $t$, and letting $t\to\infty$ gives
$DG_{\varepsilon,\tau}(A)[A]\geq F_\varepsilon(A)+\tau\tr A$.
This proves the other bound.

\subsection{The global resolvent equation}\label{sec:resolvent}

Use the data of Lemma~\ref{lem:source-data}. For $\sigma>0$ consider
\begin{equation}\label{eq:resolvent}
 G_{\varepsilon,\tau}(\Lev u)
   =\sigma(u-u_0)+\tau h\qquad\text{on }M.
\end{equation}
An upper touching test for a subsolution has the left side greater than
or equal to the right side; a lower touching test for a supersolution
has the reverse inequality. Thus the associated proper operator is
$\sigma(s-u_0)+\tau h-G_{\varepsilon,\tau}(\Lev\phi)$.

\begin{lemma}\label{lem:barriers}
Equation~\eqref{eq:resolvent} has ordered continuous global sub- and
supersolutions $L\leq u_0\leq U$, with
\begin{equation}\label{eq:barrier-bounds}
 \sigma(L-u_0)=-c_n\varepsilon^2,\qquad
 0\leq\sigma(U-u_0)\leq(1+A_\tau)\sqrt\sigma,
 \qquad A_\tau=\frac{2n}{4}\left(\frac1n+\tau\right).
\end{equation}
The upper barrier has Lipschitz constant at most one.
\end{lemma}
\begin{proof}
Set $L=u_0-c_n\varepsilon^2/\sigma$. Since
$F_\varepsilon(\Lev u_0)\geq0$, equations
\eqref{eq:mollifier-error} and \eqref{eq:source-data} show that $L$
is a subsolution.

For $R>0$ define
\[
 K_R(x,y)=\sqrt{R^2+d(x,y)^2},\qquad
 U_R(x)=\inf_{y\in M}\left\{u_0(y)+K_R(x,y)+\frac{A_\tau}{\sigma R}\right\}.
\]
The Lipschitz inequality for $u_0$ and $K_R\geq d$ give
\begin{equation}\label{eq:cone-bounds}
 u_0(x)+\frac{A_\tau}{\sigma R}\leq U_R(x)
 \leq u_0(x)+R+\frac{A_\tau}{\sigma R}.
\end{equation}
Each function of $x$ in this infimum has Lipschitz constant at most
one. The lower bound in \eqref{eq:cone-bounds} makes the infimum
finite, so $\Lip(U_R)\leq1$.

The Ricci lower bound $\Ric\geq0$ and distance Laplacian comparison
\cite[Theorem~3.6]{Wei} give, for tests touching from below,
\[
 \Delta K_R(\cdot,y)
 \leq\frac{R^2}{(R^2+d^2)^{3/2}}
       +\frac{2n-1}{\sqrt{R^2+d^2}}\leq\frac{2n}{R}.
\]
Here $d=d(x,y)$. At $y$ the smooth radial extension has the same
bound; at the cut locus
one uses the usual upper distance supports. By
\eqref{eq:bellman-trace} and \eqref{eq:levi-normalizations}, each cone
has left side at most $A_\tau/R$, whereas its right side is at least
$A_\tau/R$ by the Lipschitz lower bound and $h\geq0$.
Thus each cone is a supersolution. Stability of lower semicontinuous
infima, proved by choosing approximate minimizing cones at a touching
point, shows that $U_R$ is a supersolution. Put $R=\sigma^{-1/2}$ in
\eqref{eq:cone-bounds}. This gives \eqref{eq:barrier-bounds} and ordering.
\end{proof}

\begin{lemma}\label{lem:comparison}
For fixed $\varepsilon,\tau,\sigma>0$, an upper semicontinuous
subsolution $u$ and a lower semicontinuous supersolution $v$ of
\eqref{eq:resolvent} satisfy $u\leq v$ if $u-v$ is bounded above.
\end{lemma}
\begin{proof}
The derivative bound \eqref{eq:a-bounds} and
\eqref{eq:exhaustion} give a constant $P_{\varepsilon,\tau}$ such that
perturbing a test by $\delta b$ changes its Bellman value by at most
$\delta P_{\varepsilon,\tau}$. Therefore
\[
 u_\delta=u-\delta b-\delta P_{\varepsilon,\tau}/\sigma
\]
is still a subsolution. If $\sup(u-v)>0$, choose $\delta>0$ small
enough that $u_\delta-v$ has a positive global maximum. Properness of
$b$ ensures attainment. Work in a relatively compact coordinate ball
about a maximizing point, and add a smaller quadratic localization
penalty if necessary.

In real coordinates the Levi matrix of a quadratic jet $(p,N)$ has
the form
\begin{equation}\label{eq:coordinate-operator}
 V(x)N V(x)^*+D(x)p,
\end{equation}
where $p\in\R^{2n}$ is the first derivative, $N$ is a real symmetric
$2n\times2n$ matrix, $V(x)$ is a smooth complex $n\times2n$ matrix,
and $D(x):\R^{2n}\to\Herm(n)$ is a smooth family of linear maps.
The term $D(x)p$ comes from the connection. For matrices
$X,Y$ in the upper doubling inequality, congruence by the stacked rows
$V(x),V(y)$ gives
\[
 V(x)XV(x)^*-V(y)YV(y)^*
 \leq3\alpha[V(x)-V(y)][V(x)-V(y)]^*.
\]
The drift difference at $p=\alpha(x-y)$ is bounded by
$C\alpha|x-y|^2$. Monotonicity and \eqref{eq:a-bounds} consequently
give the structural modulus $C(\alpha|x-y|^2+|x-y|)$ for the proper
operator, including the smooth source. The Euclidean viscosity
comparison theorem \cite[Theorem~3.3]{CIL} applies. Its strictly
positive zeroth order coefficient $\sigma$ excludes the positive maximum.
This contradiction proves the assertion.
\end{proof}

\begin{proposition}\label{prop:resolvent}
For every $0<\varepsilon\leq1$, $\tau>0$, and $\sigma>0$, there is a
smooth solution of \eqref{eq:resolvent} satisfying
\begin{equation}\label{eq:resolvent-bounds}
 -c_n\varepsilon^2\leq\sigma(u-u_0)
 \leq(1+A_\tau)\sqrt\sigma.
\end{equation}
\end{proposition}
\begin{proof}
Take the supremum of all upper semicontinuous subsolutions between
$L$ and $U$ from Lemma~\ref{lem:barriers}. The family is nonempty and
closed under finite maxima. The upper regularization of its supremum
is a subsolution by local jet stability
\cite[Proposition~4.3]{CIL}. If the lower regularization failed the
supersolution condition at a lower touching test, subtract a sufficiently
small quadratic function and raise the test slightly. Strictness of
the failed inequality makes this a local subsolution. Its maximum with
a nearby member of the Perron family agrees with that member near the
boundary of a small ball but exceeds the supremum at an interior point.
If this lift could cross $U$, local comparison with the supersolution
$U$ first prevents the crossing. Gluing therefore contradicts the
definition of the supremum. The lower regularization is a supersolution.
Both envelopes have bounded difference from $u_0$ by
\eqref{eq:barrier-bounds}. Lemma~\ref{lem:comparison} identifies them,
giving a continuous viscosity solution and \eqref{eq:resolvent-bounds}.

We detail the regularity step because the operator in
\eqref{eq:coordinate-operator} depends on the first derivative.
On a smaller coordinate ball it is uniformly elliptic in $N$,
Lipschitz in $p$, and has smooth normalized spatial coefficients.
The interior $C^{1,\alpha}$ theorem for uniformly elliptic equations
with bounded linear gradient dependence applies
\cite[Theorem~1.1 and Remarks~1.3, 3.4]{Nornberg}. Treating the
H\"older continuous function $Du(x)$ as a coefficient gives an equation
whose differential dependence is through the Hessian. It is concave in
$N$ and has H\"older spatial oscillation
bounded by $C|x-y|^\alpha(\norm{N}+1)$. Its right side is H\"older
continuous. The interior theorem of Tosatti, Wang, Weinkove, and Yang
\cite[Theorem~3.4 and the following paragraph]{TWWY} gives $C^{2,\beta}$
regularity for the continuous viscosity solution. The equation is
therefore classical. Smooth elliptic bootstrapping
\cite[Section~17.5, Lemma~17.16]{GT} makes the solution smooth.
\end{proof}

\subsection{Two gradient estimates}\label{sec:gradient}

Let $u$ be the solution of Proposition~\ref{prop:resolvent}. At an
adapted frame define its linearization by
\[
 L\psi=DG_{\varepsilon,\tau}(\Lev u)[\Lev\psi].
\]
Unitary invariance makes this a globally defined second order operator.
It is increasing in the real Hessian and is nonpositive on a smooth
function at a local maximum. Its coefficients are uniformly elliptic
for fixed $\varepsilon,\tau$. Write
\[
 w=|\nabla u|^2,\qquad Q=|\nabla^2u|^2.
\]

\begin{lemma}\label{lem:bochner}
Under nonnegative holomorphic bisectional curvature,
\begin{equation}\label{eq:bochner}
 Lw\geq 2\langle\nabla G_{\varepsilon,\tau}(\Lev u),\nabla u\rangle
           +\frac\tau2 Q+\frac\tau2\Ric(\nabla u,\nabla u).
\end{equation}
Consequently the data in \eqref{eq:source-data} give
\begin{equation}\label{eq:sharp-differential}
 Lw\geq\left(2\sigma+\frac{\tau r}{2}\right)(w-\sqrt w)
              +\frac\tau2 Q.
\end{equation}
\end{lemma}
\begin{proof}
Choose a local adapted orthonormal frame whose covariant derivative
vanishes at the point in question, and diagonalize
$DG_{\varepsilon,\tau}(\Lev u)$ there by a constant unitary change.
Differentiate $w$ twice. The terms in which both derivatives fall on
the first derivative of $u$ are nonnegative quadratic forms in its
Hessian. The part from $\tau\tr\Lev=\tau\Delta/4$ is exactly
$\tau Q/2$. Commuting the remaining third derivatives gives
$2\langle\nabla G_{\varepsilon,\tau}(\Lev u),\nabla u\rangle$
and the curvature contraction with the positive coefficient matrix.
The latter is a sum of nonnegative bisectional curvatures, weighted
by the nonnegative eigenvalues of $DG_{\varepsilon,\tau}-\tau I$,
together with $\tau\Ric(\nabla u,\nabla u)/2$. This proves
\eqref{eq:bochner}. The frame is normal at the point, and its unitary change is constant,
so the calculation gives the tensorial inequality \eqref{eq:bochner}.

Differentiating \eqref{eq:resolvent}, and using
$|\nabla u_0|\leq1$, gives
\[
 2\langle\nabla G_{\varepsilon,\tau}(\Lev u),\nabla u\rangle
 \geq2\sigma(w-\sqrt w)-2\tau|\nabla h|\sqrt w.
\]
The last two terms in \eqref{eq:source-data} imply
$\Ric(\nabla u,\nabla u)\geq rw$ and
$2|\nabla h|\sqrt w\leq r\sqrt w/2$. Substitution proves
\eqref{eq:sharp-differential}.
\end{proof}

\begin{lemma}\label{lem:gradient-bounded}
For fixed $\varepsilon,\tau,\sigma>0$, the gradient of the solution
of Proposition~\ref{prop:resolvent} is bounded on $M$.
\end{lemma}
\begin{proof}
Put $v=u-u_0$. Equation~\eqref{eq:resolvent-bounds} gives $|v|\leq V$
for a finite constant $V$. Compact support of $h$ gives finite bounds
for $h$ and $|\nabla h|$. Concavity at $\Lev u$ gives
\[
 Lv\leq G_{\varepsilon,\tau}(\Lev u)
          -G_{\varepsilon,\tau}(\Lev u_0)\leq D
\]
for a finite $D\geq0$, since
$G_{\varepsilon,\tau}(\Lev u_0)\geq-c_n\varepsilon^2$.
Dropping nonnegative terms from \eqref{eq:bochner} gives
\begin{equation}\label{eq:weak-gradient}
 Lw\geq\frac\tau2 Q-C\sqrt w
\end{equation}
with $C=2\sigma+2\tau\sup|\nabla h|$.

Set $B=V+1$, $S=2V+1$, and $\Psi=(B-v)^2$. Then
$1\leq B-v\leq S$ and $0\leq\Psi\leq S^2$. The chain rule and
\eqref{eq:a-bounds} show that
\[
 L\Psi=-2(B-v)Lv+2DG_{\varepsilon,\tau}(\Lev u)
                   [\Lev(dv\otimes dv)]
 \geq-2SD+\frac\tau2|\nabla v|^2.
\]
Here $\Lev(dv\otimes dv)$ means the projection in
\eqref{eq:levi-matrix}; its trace is $|\nabla v|^2/4$.
Since $|\nabla u_0|\leq1$,
$|\nabla v|^2\geq w/2-1$, and hence
\begin{equation}\label{eq:quadratic-bernstein}
 L\Psi\geq\frac\tau4w-\frac\tau2-2SD.
\end{equation}

Choose the squared cutoffs $\eta_R$ after
\eqref{eq:cutoff-powers}. The complex-Hessian bound gives
$L\eta_R\geq-E/R$, for a constant $E$ independent of $R$ and $u$.
The positive linear form $DG_{\varepsilon,\tau}$ satisfies
Cauchy--Schwarz on tensors of rank one. Using
$|\nabla w|^2\leq4wQ$, followed by Young's inequality, in the product
rule for $L(\eta_Rw)$ gives
\begin{equation}\label{eq:bernstein-cutoff}
 L(\eta_Rw)\geq-C\sqrt w-
 \left(\frac ER+
 \frac{(\varepsilon^{-1}+\tau)^2D_0}{\tau R^2}\right)w.
\end{equation}
The constant $D_0$ comes from
$|\nabla\eta_R|^2\leq D_0R^{-2}\eta_R$; the Hessian energy in
\eqref{eq:weak-gradient} absorbs the cross term.
At an interior maximum of $\eta_Rw+\Psi$ with $\eta_R>0$, equations
\eqref{eq:quadratic-bernstein}--\eqref{eq:bernstein-cutoff} imply, for
sufficiently large $R$,
\begin{equation}\label{eq:quadratic-bound}
 \frac\tau8w\leq C\sqrt w+\frac\tau2+2SD.
\end{equation}
Thus $w$ at such a maximum is at most a constant $N$ independent of
$R$. If $\eta_Rw+\Psi$ never exceeds $S^2$, the same conclusion below
already holds. Otherwise its supremum is attained in the compact
support of $\eta_R$, and at its maximum
$\eta_Rw+\Psi\leq N+S^2$. For any fixed $x$, choose $R$ large enough
that $\eta_R(x)=1$ and that the error in
\eqref{eq:bernstein-cutoff} is at most $\tau/8$. We obtain
$w(x)\leq N+S^2$. The bound is independent of $x$, proving the lemma.
\end{proof}

\begin{proposition}\label{prop:sharp-gradient}
Every solution constructed in Proposition~\ref{prop:resolvent} satisfies
$|\nabla u|\leq1$ on $M$.
\end{proposition}
\begin{proof}
By Lemma~\ref{lem:gradient-bounded}, $w$ is bounded. For $\delta>0$,
$w-\delta b$ has a global maximum at a point $x_\delta$. The complex
Hessian of $b$ and \eqref{eq:a-bounds} give $Lb\leq P$ with a fixed
$P$. Therefore $Lw(x_\delta)\leq\delta P$. If
$w(x_\delta)\geq1$, then \eqref{eq:sharp-differential}, $r\geq0$,
and
$2(w-\sqrt w)\geq w-1$ give
$\sigma(w(x_\delta)-1)\leq\delta P$. The same upper bound for
$w(x_\delta)$ is automatic when it is less than one. For each $x$,
\[
 w(x)\leq w(x_\delta)+\delta b(x)
 \leq1+\delta P/\sigma+\delta b(x).
\]
Let $\delta\downarrow0$. This proves the assertion, uniformly in the
resolvent and truncation parameters.
\end{proof}

\subsection{The plurisubharmonic limit}\label{sec:limit}

Fix $\tau>0$ and let $\varepsilon_j,\sigma_j\downarrow0$. Let $u_j$
be the solutions above, and fix $o\in M$. Proposition~\ref{prop:sharp-gradient}
and completeness give $\Lip(u_j)\leq1$. Hence
$v_j=u_j-u_j(o)$ have a locally uniformly convergent subsequence, with
limit $v$ satisfying $\Lip(v)\leq1$ and $v(o)=0$. This follows by
Arzel\`a--Ascoli on an increasing compact exhaustion, since
$|v_j(x)|\leq d(o,x)$. By \eqref{eq:resolvent-bounds}, the functions
\[
 c_j(x)=\sigma_j(u_j(x)-u_0(x))
\]
converge uniformly to zero on $M$. Subtracting $u_j(o)$ has no effect
on its Levi form, so
\begin{equation}\label{eq:sequence-equation}
 G_{\varepsilon_j,\tau}(\Lev v_j)=c_j+\tau h.
\end{equation}

\begin{lemma}\label{lem:limit-tests}
If a $C^2$ function $\phi$ touches $v$ from above at $x$, then
\begin{equation}\label{eq:upper-tests}
 \Lev\phi(x)\geq0,\qquad
 \tau h(x)\leq(1/n+\tau)\tr\Lev\phi(x).
\end{equation}
If $\phi$ touches $v$ from below at $x$, then
\begin{equation}\label{eq:lower-tests}
 \det{}_+(\Lev\phi(x))\leq(\tau h(x))^n,
\end{equation}
where $\det{}_+A=\det A$ for $A\geq0$ and is zero otherwise.
\end{lemma}
\begin{proof}
Use a local real coordinate ball and the approximation of touching
jets under locally uniform convergence
\cite[Proposition~4.3]{CIL}. Both the point and the first derivative
in \eqref{eq:coordinate-operator} vary in this approximation; smoothness
of $V,D$ shows that the intrinsic Hermitian matrices converge to
$H=\Lev\phi(x)$.

For an upper jet, monotonicity and
\eqref{eq:sequence-equation} imply, in the limit,
\begin{equation}\label{eq:all-controls}
 \tau h(x)\leq\frac1n\tr(HB)+\tau\tr H
 \quad\text{for every }B>0\text{ with }\det B=1.
\end{equation}
Indeed any fixed such $B$ belongs to $\mathcal B_{\varepsilon_j}$
eventually, and the convolution error tends to zero. If $H$ had a
negative eigenvalue, take the corresponding eigenvalue of $B$ equal
to $t$ and all others equal to $t^{-1/(n-1)}$. Since $n\geq2$,
the right side of \eqref{eq:all-controls} tends to $-\infty$ as
$t\to\infty$. Thus $H\geq0$. Taking $B=I$ gives the trace inequality.

For a lower jet there is nothing to prove unless $H$ is positive
definite: semidefinite singular matrices have determinant zero, and
$\det{}_+H$ is zero when $H$ is not nonnegative.
For $H>0$, the approximating matrices are uniformly positive for large
$j$. The formula for the minimum over matrices with determinant one
\[
 \inf_{B>0,\det B=1}\frac1n\tr(HB)=(\det H)^{1/n}
\]
follows from the arithmetic--geometric mean inequality. Equality holds at
\[
 B=(\det H)^{1/n}H^{-1}.
\]
The same formula and its local continuity
identify the limit of the truncated minima. The lower-test inequality
in \eqref{eq:sequence-equation} therefore gives
\[
 (\det H)^{1/n}+\tau\tr H\leq\tau h(x).
\]
Drop the nonnegative trace
term and raise to the $n$th power. This proves \eqref{eq:lower-tests}.
\end{proof}

\begin{proposition}\label{prop:finite-ma}
Under the hypotheses of Lemma~\ref{lem:source-data}, the limit $v$ is
continuous PSH, is not pluriharmonic, and satisfies
\begin{equation}\label{eq:finite-ma}
 0\leq(\ddc v)^n\leq2^n n!(\tau h)^n\vol.
\end{equation}
In particular, its total Monge--Amp\`ere mass is finite.
\end{proposition}
\begin{proof}
In compatible holomorphic coordinates the covariant projected Hessian
is the usual complex Hessian, up to the metric change of frame. The
first part of \eqref{eq:upper-tests} implies the standard PSH submean
condition \cite[Proposition~1.3]{EGZ}. The continuous Bedford--Taylor
measure is defined by local smooth PSH approximation
\cite[Section~2]{BT}. On any relatively compact coordinate ball, $v$
is bounded. The local envelope theorem for PSH functions
\cite[Theorem~3.9]{GLZ}, applied to \eqref{eq:lower-tests}, bounds its
normalized Monge--Amp\`ere measure by $(\tau h)^n\vol$. The PSH envelope of $v$ equals $v$, since $v$ is already PSH.

Under a holomorphic change of coordinates, the complex Hessian
determinant is multiplied by the squared modulus of the complex
Jacobian. The metric density is multiplied by the same factor.
The ratio is therefore intrinsic. The covariance of the continuous
Bedford--Taylor extension follows from its uniqueness and the smooth
change-of-variables formula. Local measures thus agree on overlaps
and define one global measure. Formula~\eqref{eq:levi-normalizations}
gives the factor $2^n n!$ in \eqref{eq:finite-ma}. Since $h$
is smooth and compactly supported, the right side has finite mass.

If $v$ were pluriharmonic, it would be smooth, and $\phi=v$ would be
an upper test at every point with zero Levi matrix. At a point where
$h>0$, the second inequality of \eqref{eq:upper-tests} is impossible.
\end{proof}

\section{Finiteness of the Ricci integral}\label{sec:ricci-mass}

We first regularize the function of Proposition~\ref{prop:finite-ma} by
the scalar heat equation. A holomorphic canonical section then gives
weighted estimates that imply finiteness of the Ricci integral in
Proposition~\ref{prop:finite-ricci-criterion}.

\subsection{Strict heat regularization}\label{sec:smoothing}

Let $P_t$ be the scalar heat semigroup for the real Laplacian. On a
complete manifold with $\Ric\geq0$, its kernel $H(t,x,y)$ is
nonnegative, symmetric, and satisfies $\int_M H(t,x,y)\,d\mu_g(y)=1$.
We use $B_x(r)$ to denote the open metric ball with centre $x$ and radius $r$. The Gaussian heat-kernel and
volume comparison estimates imply
\begin{equation}\label{eq:second-moment}
 \int_Md(x,y)^2H(t,x,y)\,d\mu_g(y)\leq C_{\mathrm h}t.
\end{equation}
These standard scalar facts follow, for example, from
\cite[Section~3 and Theorems~6.4--6.5]{Grigoryan}. To see the last
implication, split $M$ into annuli of radii $2^j\sqrt t$, use the
Gaussian upper bound on each annulus, and use volume doubling to
compare its volume with $\mu_g(B_x(\sqrt t))$. The resulting series
is bounded by $Ct\sum_{j\geq0}2^{(2+2n)j}e^{-c4^j}$, uniformly in
$x,t$. Near the diagonal the estimate follows from stochasticity.

If $\Lip(v)\leq1$, Cauchy--Schwarz with respect to the probability
measure $H(t,x,y)\,d\mu_g(y)$ gives
\begin{equation}\label{eq:heat-comparison}
 |P_tv(x)-v(x)|\leq\int d(x,y)H(t,x,y)\,d\mu_g(y)
 \leq\sqrt{C_{\mathrm h}t}.
\end{equation}
The same moment bound makes the average well defined even though $v$
need not be bounded.

\begin{proposition}\label{prop:strict-approximation}
For the function $v$ of Proposition~\ref{prop:finite-ma}, on the
simply connected noncompact nonflat manifold there are smooth strictly
PSH functions $\phi_j$, indexed by $j=0,1,\ldots$, and $C<\infty$ such that
\begin{equation}\label{eq:strict-approximation}
 |\nabla\phi_j|\leq1,\qquad
 \phi_j\longrightarrow v\text{ locally uniformly},\qquad
 |\phi_j-v|\leq C\text{ on }M.
\end{equation}
\end{proposition}
\begin{proof}
The theorem on the kernel of the Levi form and the heat gradient
estimate of Ni and Tam
\cite[Theorem~3.1 and Lemma~1.4]{NiTam03} apply to the Lipschitz PSH
function $v$. On a sufficiently small positive time interval the
kernel of the smooth Levi form has constant complex rank and is
parallel. Complex rank $n$ gives
a pluriharmonic initial limit, excluded by
Proposition~\ref{prop:finite-ma}. A rank strictly between zero and
$n$ contradicts Lemma~\ref{lem:no-product}. Thus $P_tv$ is strictly
PSH for $0<t\leq T$ after reducing $T$ if necessary. Set
$\phi_j=P_{T/(j+1)}v$. Local initial convergence and the gradient
estimate give the first two assertions in
\eqref{eq:strict-approximation}; \eqref{eq:heat-comparison} gives the
last, with $C=\sqrt{C_{\mathrm h}T}$.
\end{proof}

\subsection{Holomorphic canonical sections}

\begin{lemma}\label{lem:canonical-weight}
Let $K_M=\Lambda^{n,0}T^*M$ be the canonical bundle with the metric
induced by $g$. There are a nonzero holomorphic section $s$ of $K_M$ and
$q>0$ such
that, for $a_j=|s|^2e^{-q\phi_j}$,
\begin{equation}\label{eq:canonical-budget}
 \sup_j\int_Ma_j\vol<\infty,\qquad
 |s|^2>0\quad\text{almost everywhere}.
\end{equation}
Each $a_j$ is smooth and nonnegative, including at the zeros of $s$.
\end{lemma}
\begin{proof}
The weighted canonical-section construction in
\cite[proof of Lemma~4.2(ii), equation~(4.3)]{NiTam03} gives $s\not\equiv0$
and $q>0$ with $\int|s|^2e^{-q\phi_0}\vol<\infty$, since $\phi_0$
is smooth strictly PSH and has at most linear growth. The nonzero
holomorphic section has a zero set of measure zero. By
\eqref{eq:strict-approximation}, $\phi_0-\phi_j\leq2C$, so
$a_j\leq e^{2qC}|s|^2e^{-q\phi_0}$. This proves the uniform integral bound in \eqref{eq:canonical-budget}.
Smoothness follows from the smooth canonical metric and the local
holomorphic coefficients.
\end{proof}

\begin{lemma}\label{lem:energy}
Let $\phi$ be smooth PSH and let $f=|s|^2e^{-q\phi}$ for a
holomorphic canonical section. Then
\begin{equation}\label{eq:energy}
 E=\ddc f-f\rho+qf\ddc\phi\geq0,\qquad
 fE=df\wedge d^cf.
\end{equation}
The same identities hold after multiplying $f$ by a positive constant
and $E$ by the same constant.
\end{lemma}
\begin{proof}
In a holomorphic coordinate patch write $s=\sigma\,dz_1\wedge\cdots
\wedge dz_n$ and $|s|^2=e^\ell|\sigma|^2$. The canonical metric
curvature formula gives $\ddc\ell=\rho$
\cite[equations~(4.63)--(4.67)]{Ballmann}; a fixed dimensional factor
in the canonical norm has zero logarithmic derivative. Put
$\psi=\ell-q\phi$. Direct differentiation of
$f=e^\psi|\sigma|^2$ and the Cauchy--Riemann equations gives
\[
 E=\ii e^\psi
     (\partial\sigma+\sigma\partial\psi)
       \wedge\overline{(\partial\sigma+\sigma\partial\psi)}.
\]
This is nonnegative. Multiplying by $f$ and differentiating $f$ once
gives $fE=df\wedge d^cf$. Both computations use smooth local
coefficients and remain valid at $\sigma=0$. The tensorial expressions agree on
overlaps. Finally, both sides of the second identity scale
quadratically, and the first identity scales linearly.
\end{proof}

\subsection{Weighted mixed estimates}\label{sec:bezout}

The identity in Lemma~\ref{lem:energy} gives two positive forms whose
mixed integrals can be estimated by compact integration by parts.

\subsubsection{A positive form from the logarithmic weight}
Let $\phi$ be smooth PSH with $|\nabla\phi|\leq L$, put $A=\ddc\phi$,
and suppose $f\geq0$ is smooth and integrable and satisfies
\eqref{eq:energy} for some $q\geq0$. Define
\begin{equation}\label{eq:log-forms}
 V=\log(1+f),\qquad H=f-\log(1+f),\qquad B=qA+\ddc V,
 \qquad\Gamma=\ddc f+qfA.
\end{equation}
Since $\rho\geq0$, equations \eqref{eq:energy} give
\begin{equation}\label{eq:gamma}
 \Gamma=f\rho+E\geq0,\qquad df\wedge d^cf=fE\leq f\Gamma.
\end{equation}
Differentiating $V,H$ gives the two identities
\begin{align}
 B-\frac{f}{1+f}\rho
   &=\frac{q}{1+f}A+\frac{1}{(1+f)^2}E\geq0,
       \label{eq:B-positive}\\
 \ddc H-f\ddc V&=\frac{df\wedge d^cf}{1+f}\geq0.
       \label{eq:H-positive}
\end{align}
Also $0\leq V\leq f$ and $0\leq H\leq f$. Thus $B$ is smooth,
closed, and nonnegative. These assertions hold even where $f=0$.

Let $T$ be a product of $n-1$ smooth closed nonnegative $(1,1)$-forms.
Exterior positivity and the quadratic polynomial associated with
$(d a+t\,d b)\wedge d^c(a+tb)\wedge T$ give the mixed
Cauchy--Schwarz inequality. In particular, \eqref{eq:gamma},
$|\nabla\phi|\leq L$, and Young's inequality imply, as inequalities between top degree forms,
\begin{equation}\label{eq:weighted-young}
 |df\wedge d^c\phi\wedge T|
 \leq\delta\Gamma\wedge T+\frac{L^2}{4\delta}f\omega\wedge T
 \qquad(\delta>0).
\end{equation}
Here and below an absolute value of a top form means the absolute value
of its coefficient against $\vol$. One can verify the constant by
symmetrizing the mixed cross term, applying the discriminant inequality,
and using $d\phi\wedge d^c\phi\leq L^2\omega$.

\begin{lemma}\label{lem:weighted-reduction}
For a fixed finite range of integers $N\geq3$, there is a constant
$C=C(q,L,N)$, depending also on the fixed cutoff family but independent
of $R\geq1$, $f$, $\nabla^2\phi$, and $T$, such that
\begin{align}
 \int\chi_R^N f A\wedge T
    &\leq C\int\chi_R^{N-2}f\omega\wedge T,
           \label{eq:A-step}\\
 \int\chi_R^N f B\wedge T
    &\leq C\int\chi_R^{N-2}f\omega\wedge T.
           \label{eq:B-step}
\end{align}
All integrals on both sides are compactly supported.
\end{lemma}
\begin{proof}
Write $I_N=\int\chi_R^NfA\wedge T$ and
$J_N=\int\chi_R^Nf\omega\wedge T$. Compact Stokes and $dT=0$ give
\[
 I_N=-\int\chi_R^Ndf\wedge d^c\phi\wedge T
       -\int f\,d(\chi_R^N)\wedge d^c\phi\wedge T.
\]
Apply \eqref{eq:weighted-young} to the first term and the metric mixed
cross estimate to the second. By \eqref{eq:cutoff-powers},
\begin{equation}\label{eq:A-preabsorb}
 I_N\leq\delta\int\chi_R^N\Gamma\wedge T
       +\frac{L^2}{4\delta}J_N+\frac{CNL}{R}J_{N-1}.
\end{equation}
Two integrations by parts, justified by compact support, yield
\[
 \int\chi_R^N\Gamma\wedge T
   =qI_N+\int f\ddc(\chi_R^N)\wedge T
   \leq qI_N+\frac{C_N}{R}J_{N-2}.
\]
As $0\leq\chi_R\leq1$, $J_N,J_{N-1}\leq J_{N-2}$.
Choose $\delta=[2(q+1)]^{-1}$ so that $\delta q\leq1/2$ and absorb
the $I_N$ term in \eqref{eq:A-preabsorb}. This proves
\eqref{eq:A-step}, with for example a constant bounded by
$2(\delta C_N+L^2/(4\delta)+CNL)$ for $R\geq1$.

For \eqref{eq:B-step}, use \eqref{eq:H-positive} and then compact
integration by parts:
\[
 \int\chi_R^N fB\wedge T
 \leq qI_N+\int\chi_R^N\ddc H\wedge T
 =qI_N+\int H\ddc(\chi_R^N)\wedge T.
\]
The last term is at most $C_NJ_{N-2}/R$ because $0\leq H\leq f$.
Equation~\eqref{eq:A-step} now proves the result. The primitives in these integrations by parts are smooth and
compactly supported, so Stokes' theorem
\cite[Theorem~16.11]{LeeSM} applies.
\end{proof}

\subsubsection{Induction on the number of factors}
Iterate Lemma~\ref{lem:weighted-reduction} with
$T=A^a\wedge B^b\wedge\omega^{n-1-a-b}$. Every form involved is
closed and nonnegative. Increasing the original cutoff exponent by
two for each factor removed gives
\begin{equation}\label{eq:mixed-weighted-bound}
 \int\chi_R^{N+2(a+b)} fA^a\wedge B^b\wedge\omega^{n-a-b}
 \leq C^{a+b}\int\chi_R^N f\omega^n
 \leq C^{a+b}n!\int_M f\vol,
\end{equation}
provided the finitely many intermediate exponents are at least three.
The proof is induction on $a+b$: remove one $A$ or $B$ factor using
the corresponding inequality and then apply the induction hypothesis.
The constants in a finite exponent range are fixed in advance.

\begin{proposition}\label{prop:top-comparison}
With $N=2n+5$ there is $D=D(n,q,L,\chi)$ such that
\begin{equation}\label{eq:bezout-compact}
 \int\chi_R^N\left(\frac{f}{1+f}\right)^n\rho^n
 \leq q^n\int\chi_R^N A^n+\frac{D}{R}\int_M f\vol
 \qquad(R\geq1).
\end{equation}
\end{proposition}
\begin{proof}
By \eqref{eq:B-positive} and positivity of mixed products,
$(f/(1+f))^n\rho^n\leq B^n$. Expand the difference by replacing its
$n$ factors one at a time:
\[
 B^n-q^nA^n
 =\ddc V\wedge\sum_{j=0}^{n-1}B^j\wedge(qA)^{n-1-j}.
\]
After multiplication by $\chi_R^N$, each summand integrates to
\[
 \int V\ddc(\chi_R^N)\wedge B^j\wedge(qA)^{n-1-j}.
\]
Since $0\leq V\leq f$, its absolute value is bounded by
\[
 \frac{C_Nq^{n-1-j}}R
  \int\chi_R^{N-2}f\omega\wedge B^j\wedge A^{n-1-j}.
\]
Apply \eqref{eq:mixed-weighted-bound} to the $n-1$ factors. The
remaining cutoff exponent is $N-2-2(n-1)=5$, so all preceding
integrations by parts are justified. Summing the finitely many bounds
proves \eqref{eq:bezout-compact}.
\end{proof}

\subsection{Passage to the limit}

\begin{proposition}\label{prop:finite-ricci-criterion}
Let $M$ be a complete connected noncompact K\"ahler manifold with
nonnegative holomorphic bisectional curvature and $n\geq2$. Suppose a
continuous PSH function $v$ satisfies $\Lip(v)\leq1$, has finite total
Monge--Amp\`ere mass, and
admits smooth PSH approximations satisfying
\eqref{eq:strict-approximation}, with $\phi_0$ strictly PSH. Then
\begin{equation}\label{eq:finite-ricci}
 \int_M\rho^n<\infty.
\end{equation}
\end{proposition}
\begin{proof}
Take $s,q$ as in Lemma~\ref{lem:canonical-weight}. For a sequence
$\delta_k\downarrow0$, set
\[
 f_{k,j}=\delta_k^{-1}|s|^2e^{-q\phi_j},\qquad
 a=|s|^2e^{-qv}.
\]
The function $a$ is continuous and positive almost everywhere. For
each fixed $k$, the weights $f_{k,j}$ have a common finite $L^1$ bound
$F_k$ by \eqref{eq:canonical-budget}. Lemma~\ref{lem:energy} supplies
the scaled energy identities for each $k,j$. Apply
Proposition~\ref{prop:top-comparison} to obtain
\begin{equation}\label{eq:ordered-estimate}
 \int\chi_R^N\left(\frac{f_{k,j}}{1+f_{k,j}}\right)^n\rho^n
 \leq q^n\int\chi_R^N(\ddc\phi_j)^n+\frac{DF_k}{R}.
\end{equation}
First fix $k,R$ and let $j\to\infty$. On the compact support of
$\chi_R$, the left integrand is dominated by the smooth nonnegative
Ricci top density, and its weight converges to
$[a/(a+\delta_k)]^n$. The right integrals converge by locally uniform
Bedford--Taylor continuity. Thus
\[
 \int\chi_R^N\left(\frac{a}{a+\delta_k}\right)^n\rho^n
 \leq q^n\int\chi_R^N(\ddc v)^n+\frac{DF_k}{R}.
\]
Now let $R\to\infty$, with $k$ fixed. Fatou on the left,
$0\leq\chi_R\leq1$, and the assumed finite mass on the right give
\[
 \int\left(\frac{a}{a+\delta_k}\right)^n\rho^n
 \leq q^n\int_M(\ddc v)^n.
\]
Finally $a>0$ almost everywhere, so Fatou as $k\to\infty$ proves
\begin{equation}\label{eq:ricci-mass-bound}
 0\leq\int_M\rho^n\leq q^n\int_M(\ddc v)^n<\infty.
\end{equation}
Each $F_k$ is finite, and the limit in $R$ is taken before the limit
in $k$. The Ricci integrals in these limits are understood as
nonnegative extended integrals until \eqref{eq:ricci-mass-bound} proves
their finiteness.
\end{proof}

In particular, Propositions~\ref{prop:finite-ma} and
\ref{prop:strict-approximation} prove \eqref{eq:finite-ricci} on the
simply connected noncompact nonflat manifold under consideration.

\section{Critical integrability and vanishing of closed forms}\label{sec:critical}

\subsection{A critical Stokes identity}
We first prove the Stokes identity used to compare the top integrals
of two closed forms that differ by $\ddc F$.

\begin{lemma}\label{lem:critical-stokes}
Let $M$ be complete noncompact K\"ahler with nonnegative bisectional
curvature and $n\geq2$. Let $\alpha_1,\ldots,\alpha_{n-1}$ be smooth
closed nonnegative real $(1,1)$-forms with $|\alpha_j|\in L^n(M)$.
Suppose $F$ is smooth, $|\nabla F|\in L^n(M)$, and
$\ddc F\wedge\alpha_1\wedge\cdots\wedge\alpha_{n-1}$ is integrable.
Then
\begin{equation}\label{eq:critical-stokes}
 \int_M\ddc F\wedge\alpha_1\wedge\cdots\wedge\alpha_{n-1}=0.
\end{equation}
\end{lemma}
\begin{proof}
Write $T=\alpha_1\wedge\cdots\wedge\alpha_{n-1}$ and use
\eqref{eq:cutoffs}. Compact Stokes gives
\[
 \int\chi_R\ddc F\wedge T=-\int d\chi_R\wedge d^cF\wedge T.
\]
The mixed Cauchy--Schwarz inequality and
\eqref{eq:form-order-bound} bound the absolute density on the right by
\[
 \frac{C_n}{R}|\nabla F|\prod_{j=1}^{n-1}|\alpha_j|.
\]
This product has exactly $n$ factors in $L^n$, and is therefore
integrable by H\"older. Its integral times $C_n/R$ tends to zero.
The assumed integrability of the left density permits dominated
convergence as $\chi_R\to1$, proving \eqref{eq:critical-stokes}.
\end{proof}

\begin{lemma}\label{lem:top-invariance}
Let $M$ satisfy the geometric hypotheses of
Lemma~\ref{lem:critical-stokes}, and let $A,B$ be smooth closed
nonnegative real $(1,1)$-forms with
$|A|,|B|\in L^n(M)$. If $A-B=\ddc F$ for a smooth $F$ with
$|\nabla F|\in L^n(M)$, then
\begin{equation}\label{eq:top-invariance}
 \int_M A^n=\int_M B^n.
\end{equation}
\end{lemma}
\begin{proof}
Every product of $n$ factors chosen from $A,B$ is integrable by
\eqref{eq:mixed-norm} and H\"older. Expand
$A^n-B^n=(A-B)\wedge\sum_{j=0}^{n-1}A^j\wedge B^{n-1-j}$.
Each summand is integrable as a difference of two such products.
Lemma~\ref{lem:critical-stokes} applied to its $n-1$ remaining factors
shows that its integral is zero. This proves the assertion.
\end{proof}

\subsection{Scalar heat estimates}
The following two lemmas concern the scalar heat semigroup on a complete
connected noncompact manifold with $\Ric\geq0$.

\begin{lemma}\label{lem:scalar-decay}
For every integer $p\geq2$ and $b\in L^p(M)$,
\begin{equation}\label{eq:scalar-decay}
 \norm{P_tb}_{L^p}\longrightarrow0\quad(t\to\infty).
\end{equation}
\end{lemma}
\begin{proof}
Yau's infinite volume theorem \cite{Yau} implies that $M$ has infinite
volume. The $L^2$ kernel of the
nonnegative scalar Laplacian consists of functions with zero weak
gradient, hence of constants by connectedness. No nonzero such
constant is in $L^2$. The spectral formula for the heat operator
\cite[Section~3]{Grigoryan} and dominated convergence therefore give
$\norm{P_tg}_2\to0$ for $g\in L^2$. For bounded $g\in L^2\cap L^p$,
Markov contractivity yields
\[
 \norm{P_tg}_p^p
 \leq\norm{g}_\infty^{p-2}\norm{P_tg}_2^2\longrightarrow0.
\]
Bounded functions of compact support are dense in $L^p$. Given
$b\in L^p$, approximate it by such a $g$. Since $P_t$ is an $L^p$
contraction,
$\norm{P_tb}_p\leq\norm{b-g}_p+\norm{P_tg}_p$. Take the limit in
$t$ and then improve the approximation. This proves
\eqref{eq:scalar-decay}.
\end{proof}

\begin{lemma}\label{lem:integrated-gradient}
If $p\geq2$ is an integer and $f\in C^\infty(M)\cap L^p(M)$, then, for every
$t>0$, the function
\begin{equation}\label{eq:Ft}
 F_t=2\int_0^tP_sf\,ds
\end{equation}
is smooth and $|\nabla F_t|\in L^p(M)$.
\end{lemma}
\begin{proof}
The scalar reverse Poincar\'e inequality and Jensen give
\cite[equation~(2.1) and the following estimate]{Ledoux}
\[
 \norm{\nabla P_sf}_p\leq(2s)^{-1/2}\norm{f}_p\qquad(p\geq2).
\]
Minkowski's inequality, first on $[\epsilon,t]$ and then as
$\epsilon\downarrow0$,
gives a weak gradient in $L^p$ with
\begin{equation}\label{eq:Ft-gradient}
 \norm{\nabla F_t}_p
 \leq2\int_0^t(2s)^{-1/2}\norm{f}_p\,ds<\infty.
\end{equation}
Contractivity also gives $\norm{F_t}_p\leq2t\norm{f}_p$. In the weak
sense $\Delta F_t=2(P_tf-f)$. The right side is smooth, so local
elliptic regularity makes $F_t$ smooth and identifies its weak gradient
with its ordinary metric gradient. Equivalently, this follows by smooth
local heat convergence to the smooth initial function.
\end{proof}

\subsection{Closed Hodge heat and vanishing}
For a real $(1,1)$-form let $\Delta_d=d d^*+d^*d$ be the nonnegative
real Hodge Laplacian. We use the heat equation
$\partial_t\eta=-\Delta_d\eta$, consistent with the real scalar
Laplacian in \eqref{eq:trace-real}. Here $d^*$ is the $L^2$ adjoint
of the exterior derivative.

\begin{proposition}\label{prop:critical-vanishing}
Let $M$ be a connected complete noncompact K\"ahler manifold with
nonnegative holomorphic bisectional curvature and $n\geq2$. Let
$\alpha$ be a smooth closed nonnegative real $(1,1)$-form. Suppose
that $|\alpha|$ and the smooth function $f=\Lambda_\omega\alpha$
belong to $L^n(M)$. Then
\begin{equation}\label{eq:critical-zero}
 \int_M\alpha^n=0.
\end{equation}
\end{proposition}
\begin{proof}
The closed Hodge heat theorem of Ni and Tam
\cite[Theorem~4.1 and Lemma~3.2]{NiTam13} applies to $\alpha$.
Indeed, for $R\geq1$, H\"older and the positive volume of a unit ball
give
\[
 \frac{1}{R^2\mu_g(B_o(R))}\int_{B_o(R)}|\alpha|\vol
 \leq\frac{\norm{|\alpha|}_{L^n}}
              {R^2\mu_g(B_o(1))^{1/n}}\longrightarrow0.
\]
The resulting smooth real evolution is closed and satisfies
\begin{equation}\label{eq:hodge-domination}
 |\eta_t|\leq P_t|\alpha|\qquad(t>0).
\end{equation}
Its nonnegativity follows by positivity preservation for $(1,1)$
Hodge heat under nonnegative bisectional curvature, as used in
\cite[Section~5]{NiTam13}. Tracing gives
$\Lambda_\omega\eta_t=P_tf$: both sides solve the same scalar heat
equation and lie in the $L^n$ uniqueness class by
\eqref{eq:hodge-domination}. For closed $(1,1)$-forms, the K\"ahler
identities give
\[
 -\Delta_d\eta_t=2\ddc\Lambda_\omega\eta_t.
\]
Integration from zero, using local initial convergence, yields the
exact transgression
\begin{equation}\label{eq:hodge-transgression}
 \eta_t-\alpha=\ddc F_t,\qquad
 F_t=2\int_0^tP_sf\,ds.
\end{equation}
The identity is also obtained directly from the calculation in
\cite[proof of Theorem~2.1]{NiTam13}. In particular its coefficient
is 2 for our real-Laplacian convention.

By \eqref{eq:hodge-domination} and scalar contractivity,
$|\eta_t|\in L^n$. Lemma~\ref{lem:integrated-gradient} gives
$|\nabla F_t|\in L^n$ for every fixed $t>0$. Apply
Lemma~\ref{lem:top-invariance} to \eqref{eq:hodge-transgression}:
\begin{equation}\label{eq:hodge-mass-constant}
 \int_M\eta_t^n=\int_M\alpha^n\qquad(t>0).
\end{equation}
On the other hand, positivity, \eqref{eq:mixed-norm}, and
\eqref{eq:hodge-domination} give
\[
 0\leq\int_M\eta_t^n\leq n!\norm{P_t|\alpha|}_n^n\longrightarrow0
\]
by Lemma~\ref{lem:scalar-decay}. Equation~\eqref{eq:hodge-mass-constant}
therefore proves \eqref{eq:critical-zero}.
\end{proof}

\section{Proof of the main theorem}\label{sec:conclusion}

\begin{lemma}\label{lem:finite-mass-flat}
Suppose $M$ is complete connected noncompact K\"ahler, $n\geq2$,
has nonnegative holomorphic bisectional curvature, and satisfies
\eqref{eq:pinching}. If $\int_M\rho^n<\infty$, then $M$ is flat.
\end{lemma}
\begin{proof}
Let $\lambda_1,\ldots,\lambda_n\geq0$ be the eigenvalues of $\rho$
relative to $\omega$, and let $m=\max\lambda_j$. Pinching gives
$\lambda_j\geq\kappa m$, and hence
\[
 \prod_j\lambda_j\geq\kappa^n m^n,\qquad
 |\rho|^n\leq n^{n/2}m^n,\qquad
 (\Lambda_\omega\rho)^n\leq n^n m^n.
\]
Since $\rho^n=n!\prod_j\lambda_j\vol$, the finite mass hypothesis
implies $|\rho|,\Lambda_\omega\rho\in L^n(M)$. The Ricci form is
smooth, closed, and nonnegative. Proposition~\ref{prop:critical-vanishing}
therefore gives $\int_M\rho^n=0$. Its nonnegative continuous density
must vanish everywhere: otherwise it is positive on a nonempty open
set, which has positive Riemannian volume. The first displayed
inequality forces $m=0$ at every point, so $\Ric\equiv0$.
Tracing the nonnegative bisectional curvature in any unitary frame
then gives zero in every bisectional direction. The K\"ahler curvature
symmetries and polarization imply $\Rm_g\equiv0$.
\end{proof}

\begin{proof}[Proof of Theorem~\ref{thm:main}]
Assume that $M$ is noncompact. First suppose that it is simply
connected. If it were nonflat, Lemmas~\ref{lem:no-product} and
\ref{lem:source-data}, together with Proposition~\ref{prop:resolvent} and
Lemma~\ref{lem:limit-tests},
would give the finite-mass nonpluriharmonic PSH function of
Proposition~\ref{prop:finite-ma}. Proposition~\ref{prop:strict-approximation}
and Lemma~\ref{lem:canonical-weight} apply. Thus
Proposition~\ref{prop:finite-ricci-criterion} gives
$\int_M\rho^n<\infty$. Lemma~\ref{lem:finite-mass-flat} contradicts
nonflatness. Hence every simply connected noncompact example is flat.

For general $M$, let $\pi:\widetilde M\to M$ be its connected
universal covering and pull back $g,J$. The Riemannian
covering construction \cite[Chapters~2, 5, 6]{LeeRM} gives a complete
simply connected K\"ahler manifold and a local K\"ahler isometry.
The covering is second countable, since the fundamental group of a
second-countable manifold is countable. The curvature tensor and its
Ricci trace are preserved by the differential of $\pi$. Thus
$\widetilde M$ has nonnegative bisectional curvature and satisfies
\eqref{eq:pinching} with the same $\kappa$. It is noncompact: a
continuous surjective image of a compact covering space would make
$M$ compact. The simply connected case gives $\Rm_{\widetilde g}=0$.
For any $x\in M$, choose a lift $\widetilde x$; the tangent isometry
$d\pi_{\widetilde x}$ is surjective and intertwines the curvature
tensors. Therefore $\Rm_g(x)=0$. This proves the noncompact case and
proves the theorem.
\end{proof}

\begin{remark}\label{rem:dimension}
The hypothesis $n\geq2$ is used both in the determinant-one test of
Lemma~\ref{lem:limit-tests} and in the critical scalar gradient estimate,
which is applied with $p=n\geq2$.
\end{remark}

\end{document}